\documentclass[12pt]{amsart}
\usepackage{amsthm}
\usepackage{amssymb}
\usepackage{amsmath}

\newtheorem{Thm}{Theorem}[section]
\newtheorem{Prop}{Proposition}[section]
\newtheorem{Lemma}{Lemma}[section]
\newtheorem{Cor}{Corollary}[section]
\theoremstyle{definition}
\newtheorem{definition}{Definition}[section]

\theoremstyle{remark}

\newtheorem{rem}{Remark}[section]

\newcommand{\milano}{Dipartimento di Matematica ``F. Enriques"
 \\ Universit\`a degli Studi di Milano \\ Via Saldini 50 \\ 20133
Milano, Italy}

\newcommand{\Pin}[1]{{\mathbb P}^{#1}}

\newcommand{\bxj}{\{\mathbf{X}_j\}}
\newcommand{\nbXi}{\mathbf{X}_i}

\begin{document}

\date{\today}
\title[Smooth Determinantal Varieties and Critical Loci]{Smooth Determinantal Varieties and Critical Loci in Multiview Geometry}
\author[M.Bertolini]{Marina Bertolini} 
\email{marina.bertolini@unimi.it} \thanks{The authors are members
of GNSAGA of INdAM}

\author[R.Notari]{Roberto Notari}
\email{roberto.notari@polimi.it}

\author[C.Turrini]{Cristina Turrini}

\email{cristina.turrini@unimi.it}

\address{\milano}
\address{Dipartimento di Matematica ``F.Brioschi", Politecnico di Milano, Piazza Leonardo da Vinci 32, 20133 Milano}

\begin{abstract}
Linear projections from $ \mathbb{P}^k $ to $ \mathbb{P}^h$ appear in computer vision as models of images of dynamic or segmented scenes. Given multiple projections of the same scene, the identification of many enough
correspondences between the images allows, in principle, to reconstruct the position of the projected objects. A critical locus for the
reconstruction problem is a variety in $ \mathbb{P}^k $ containing
the set of points for which the reconstruction fails. Critical loci turn out to be determinantal varieties. In this paper we determine and classify all the smooth critical loci, showing that they are classical projective varieties.
\end{abstract}



\maketitle


\section{Introduction}

In this paper we classify the smooth determinantal varieties arising in the multiview geometry and  computer vision settings as {\it{critical loci}} for reconstruction problems. Since critical loci and determinantal varieties belong to different research fields, it is mandatory to explain the relation between them.

\

Photos of static three-dimensional scenes
taken from pinhole cameras are usually modelled by
linear projections from $\Pin{3}$ to $\Pin{2}$ . Similarly, in computer vision,  linear projections from
$\Pin{k}$ to $\Pin{h},$ are used to describe videos or images of
particular dynamic and segmented scenes
(\cite{Shas-Wo,Hart-Schaf,vid4}). For this reason a {\it{camera}} can be identified with a linear projection $\pi: \mathbb{P}^k \dashrightarrow \mathbb{P}^h$.

The {\it reconstruction} problem is the following: given a set of points in $ \mathbb{P}^k$ with unknown coordinates, called {\it scene}, and $n$ images of it in $n$ target spaces $\Pin{h_i}$, $i=1...n$, taken from unknown cameras, the goal is to recover the positions of cameras and
scene points in the ambient space $\Pin{k}$ .

Sufficiently many images and sufficiently many corresponding points in the given images should in principle allow
for a successful projective reconstruction, where corresponding points in the targets are images of the same point in the scene. Anyway, there exist
sets of points, in the ambient space $\Pin{k},$ for which the
projective reconstruction fails. These configurations of points
are called {\it critical}, which means that there exist other non
projectively equivalent sets of points and cameras that give the
same images in the target spaces.

Critical loci turn out to be algebraic varieties and have been studied by many authors,
indeed there is a wide literature on the subject. With analysis ad hoc,
in the classical case of projections from $ \mathbb{P}^3$ to $ \mathbb{P}^2$ \cite{buch-88,kra,May-92,Hart1,KHA,Hart-Ka,May-Shas2}, the critical loci can be twisted cubic curves \cite{buch-88}, or quadric surfaces \cite{kra, May-92}.
In the case of projections onto $ \mathbb{P}^2 $ from higher dimensional spaces, \cite{be-tur1, LAIA, tubbCHAPTER}, critical loci have been proven to be minimal degree varieties \cite{be-tur1} for one projection, or,
in a more general setting and under suitable genericity assumptions, either hypersurfaces, if the ambient space is odd dimensional, or determinantal varieties of codimension two if the ambient space is even dimensional \cite{LAIA}.

Later, in \cite{AstromKahl,  bnt1, bbnt, BM} the study of the ideal of critical loci has been formalized making use of the so-called Grassmann tensor introduced in \cite{Hart-Schaf}. A seminal case of this approach has been considered in \cite{AstromKahl}, where the authors computed the equations of the critical locus for triples of projections from $ \mathbb{P}^2$ to $ \mathbb{P}^1$. In \cite{bnt1, bbnt} the case of three projections from $ \mathbb{P}^4$ to $ \mathbb{P}^2$ is studied in detail. When the projections are general enough, critical loci are shown to be Bordiga surfaces, and conversely, every Bordiga surface is shown to be critical for suitable triples of projections. When the genericity assumptions are not fulfilled, critical loci are shown to be not irreducible, with components of different dimensions. Finally, in \cite{BM}, critical loci which are hypersurfaces in the ambient space are investigated.

\

On the other hand, the classification of embedded smooth projective varieties is a classical problem in algebraic geometry. For low degree or dimension and codimension, the classical approach to a classification problem consists in applying suitable techniques to get a finite list of possible cases and further to construct examples for the surviving cases. Determinantal varieties are quite classical varieties, whose study takes advantage of homological techniques. The seminal result in the subject is Hilbert-Burch Theorem, but it is worth mentioning the structure theorem of codimension 3 Gorenstein ideals by D. Buchsbaum and D. Eisenbud, or Buchsbaum-Rim and Eagon-Northcott complexes.

\

Under this view point, since critical loci are in the class of determinantal varieties, in this paper we approach in full generality the problem of determining which smooth determinantal varieties appear as critical loci and of classifying them. More precisely, we determine under what assumptions the critical locus for a reconstruction problem for $n$ projections from $\mathbb{P}^k$ to $\mathbb{P}^{h_i}$ for $i=1, \dots n$ is a smooth variety, and provide a complete and effective classification of smooth varieties with codimension at least $ 2 $ that can be critical.

The classification results are summarized as follows, where $ n $ is the number of projections.

\begin{itemize}
	\item $n=2$
	
		All smooth critical loci are minimal degree varieties.
		Conversely, with the only exception of Veronese surface in $\mathbb{P}^5$, every minimal degree variety embedded in $\mathbb{P}^k$, with codimension $c$, $c \leq k \leq 2c+1$, is the critical locus for suitable pairs of projections.
	
	\item $n=3$
	
	$X$ in $\mathbb{P}^k$ is a smooth critical locus if and only if $X$ is
	\begin{itemize}
	\item a cubic plane curve, in the case of triples of projections from $\mathbb{P}^2$ to $\mathbb{P}^1$, $\mathbb{P}^1$, and $\mathbb{P}^1$;
	\item a cubic surface in $\mathbb{P}^3$, in the case of triples of projections from $\mathbb{P}^3$ to $\mathbb{P}^1$, $\mathbb{P}^1$, and $\mathbb{P}^2$;
	\item a Bordiga surface in $\mathbb{P}^4$ in the case of triples of projections from $\mathbb{P}^4$ to $\mathbb{P}^2$, $\mathbb{P}^2$, and $\mathbb{P}^2$.
	\end{itemize}
	
\item $n=4$
	
	Smooth critical loci are quartic determinantal surfaces in $\mathbb{P}^3$, in the case of $4-$uples of projections from $\mathbb{P}^3$ to $\mathbb{P}^1$, $\mathbb{P}^1$, $\mathbb{P}^1$, and $\mathbb{P}^1$, containing four pairwise skew lines. Conversely, given four pairwise skew lines, it is possible to construct a smooth determinantal surface of degree $ 4$ through them that is the critical locus for a suitable reconstruction problem as above.
\end{itemize}

The plan of the paper is as follows. In section \ref{prelimsec}, we introduce the setting of multiple view geometry and we recall the construction of the Grassmann tensor. In section \ref{crit_loc}, we introduce critical loci and determine the generators of their  ideals, showing in particular that critical loci are determinantal varieties. In section \ref{numerical}, we give some numerical bounds for critical loci to be smooth, and in particular we show that a critical locus can be smooth only if the number $n$ of projections is at most $ 4$. The remaining sections \ref{n2}, \ref{n3}, \ref{n4} are devoted to the study of critical loci in the cases $n=2, n=3, n=4$, respectively.


\section{On multiview Geometry and Grassmann tensors}
\label{prelimsec}

In this section we fix notation and terminology and give a short
overview of classical facts in Computer Vision related to the
problem of projective reconstruction of scenes and cameras from
multiple views.

In this context, a {\it camera} $P$ is a linear projection from
$\Pin{k}$ onto $\Pin{h},$  from a linear subspace $ C $ of
dimension $ k-h-1$,  called {\it center of projection}. The target
space $\Pin{h}$ is called {\it view}. A {\it scene} is a set of
points $ \nbXi \in \Pin{k}, i=1, \dots, N$.

Using homogeneous coordinates in $\Pin{k}$ and $\Pin{h},$ we
identify $P$ with a $(h+1) \times (k+1) $ matrix of maximal rank,
defined up to a multiplicative constant.  Hence $ C $ comes out to
be the right annihilator of $P.$

Let us consider a set of $ n $ cameras $P_j:\Pin{k}\setminus C_{j} \to
\Pin{h_j}$ projecting the same scene in $\mathbb{P}^k$ and the
corresponding set of images in the different target spaces. In
this setting, proper linear subspaces $L_i \subseteq \mathbb{P}^{h_i}, i=1
\dots n$, are said to be
\textit{corresponding} if there exists at least a point
$\mathbf{X} \in \Pin{k}$ such that $P_i(\mathbf{X})\in L_i $ for
all $i=1 \dots n$.

In the context of multiple view geometry, the problem of
{\it{projective reconstruction}} of a scene, given  multiple
images of it, is the following: given many enough scene points in
$\mathbb{P}^k$ and identified a suitable number of
{\it{corresponding subspaces}} on each image, one wants to get the
projection matrices (up to projective transformations), i.e. the
cameras, and the coordinates in $\mathbb{P}^k$ of the scene
points.

\subsection{The Grassmann tensors}
Hartley and Schaffalitzky,  \cite{Hart-Schaf}, have constructed a
set of multiview tensors, called {\it Grassmann tensors}, encoding
the relations between sets of corresponding subspaces. We recall
here the basic elements of their construction.

We consider $ n $ projections $P_j:\Pin{k}\setminus{C_j} \to
\mathbb{P}^{h_j},$ $j = 1, \dots, n,$ with centers $ C_{1}, \dots,
C_{n}$. \medskip

{\em First generality assumption}: we assume that the intersection $ C_1 \cap \dots \cap C_n $ is empty.
\medskip

Let $ L_j \subseteq \mathbb{P}^{h_j} $ be a general linear subspaces of
codimension $\alpha_j$, $j=1, \dots, n$. We say that $(L_1, \dots,
L_n)$ is a $n$--tuple of corresponding subspaces if and only if
$\overline{(P_1)^{-1}(L_1)} \cap \dots \cap
\overline{(P_n)^{-1}(L_n)}$ is not empty, where $ \overline{Y} $
is the Zariski closure of $Y$. We allow $ \alpha_j $ to be equal to $ 0 $ for some $ j$.
If this happens, the associated view does not impose any constrain
to the reconstruction problem, and so the effect of setting $
\alpha_j = 0 $ is to decrease the number of views.

We remark that the Computer Vision community uses a slightly
different definition of corresponding spaces: the spaces are said
to be corresponding if $ (P_1)^{-1}(L_1) \cap \dots \cap
(P_n)^{-1}(L_n) $ is not empty. The difference is that the centers
$ C_j $ are not considered when the inverse images are intersected
in this setting, while we prefer to include them, so to get projective varieties, and not only open subsets of them.

From the Grassmann formula, if $\sum_j \alpha_j = k+1,$ the
existence of points in the previous intersection gives a constrain
which allows us to construct the Grassmann tensor. Hartley and
Schaffalitzky call the $n$--tuple $(\alpha_1, \dots, \alpha_n)$ a
{\it profile} for the reconstruction problem. We remark that we allow $ \alpha_j = 0$, too, while, in \cite{Hart-Schaf}, $ \alpha_j \geq 1 $ for every $ j=1, \dots, n$.

Let $\{L_1, \dots, L_n\} $ be $n$ general linear subspaces as
above and let $S_j$ be the maximal rank matrix of type $(h_j+1)
\times (h_j - \alpha_j +1)$ whose columns are a basis for $L_j$.
By definition, if the $L_j$'s are corresponding subspaces, there
exists a point $\mathbf{Y} \in \Pin{k}$ such that
$P_j(\mathbf{Y})\in L_j$ for every $j$. In other words there exist
$n$ vectors $\mathbf{v_j} \in \mathbb{C}^{h_j -\alpha_j +1}$ $j =
1,\dots,n$ such that:
\begin{equation}
\label{grasssystem}
\left(%
\begin{array}{cccccc}
  P_1 & S_1 & 0 & \dots & 0 & 0 \\
  P_2 & 0 & S_2 & \dots & 0 & 0 \\
  \vdots & & & & & \vdots \\
  P_n & 0 & 0 & \dots & 0 & S_n \\
\end{array}%
\right)
\left(%
\begin{array}{c}
  \mathbf{Y} \\
  \mathbf{v_1} \\
  \mathbf{v_2} \\
  \vdots \\
  \mathbf{v_n} \\
\end{array}%
\right)=\left(%
\begin{array}{c}
  0 \\
  0 \\
  \vdots \\
  0 \\
\end{array}%
\right).
\end{equation}

The coefficient matrix $ T^{P_1, \dots, P_n}_{S_1 ,
\dots, S_n}$ is square of order $ n +
\sum h_j = k + 1 + \sum (h_j - \alpha_j + 1)$, where the left side
is the number of rows, the right side is the number of columns,
and they coincide due to our assumptions on the profile.
The existence of a non--trivial solution $(\mathbf{Y},
\mathbf{v_1},\dots,\mathbf{v_n}) $ of system (\ref{grasssystem})
implies that the determinant of $ T^{P_1, \dots, P_n}_{S_1 ,
\dots, S_n}$ is zero.

The determinant $
\mathcal{T}^{P_1, \dots, P_n}(L_1, \dots, L_n) = \det(T^{P_1, \dots, P_n}_{S_1, \dots, S_n}) $ can be thought of
as a $n$--linear form (tensor) in the Pl\"{u}cker coordinates of
the spaces $L_j$'s, in the corresponding Grassmann variety. This
tensor is called {\it Grassmann tensor}. From the above
discussion, it follows that this tensor vanishes if and only if
the linear spaces $ L_1, \dots, L_n $ are corresponding. In \cite{Hart-Schaf}, the authors show that the Grassmann tensor allows the reconstruction of the projection matrices, up to the only case when all target spaces are $ \mathbb{P}^1$. For this reason, the Computer Vision community does not consider the case above in reconstruction problems.
\medskip


\section{Critical loci and their ideals}
\label{crit_loc}
Roughly speaking, one guesses that the reconstruction problem can
be successfully solved if sufficiently many views and sufficiently
many sets of corresponding points in the given views are known.
This is generally true, but even in the classical set--up of two
projections from $\Pin{3}$ to $\Pin{2}$ one can have non
projectively equivalent pairs of scenes and
cameras that produce the same images in the view planes, thus
preventing reconstruction. Such configurations and the loci they
describe are referred to as {\it critical}. In \cite{LAIA},
critical loci for projective reconstruction of camera centers and
scene points from multiple views for projections from $\Pin{k}$ to
$\Pin{2}$ have been introduced and studied.

Now we recall the basic definition.

\begin{definition}
\label{critconfm} Given $ n $ projections $ Q_j: \mathbb{P}^k \dashrightarrow \mathbb{P}^{h_j}$,
a set of points $\{ \mathbf{X}_1, \dots, \mathbf{X}_N\}$
in $\Pin{k}$ is said to be a \textit{critical
configuration} for projective reconstruction for $ Q_1, \dots, Q_n$ if
there exists another set of $ n $ projections $P_i : \mathbb{P}^k \dashrightarrow \mathbb{P}^{h_i} $
and another set
$\{\mathbf{Y}_1, \dots, \mathbf{Y}_N \} \subset \Pin{k}$,
non-projectively equivalent to $ \{ \mathbf{X}_1, \dots,
\mathbf{X}_N \},$ such that, for all $i = 1, \dots, n$ and $j = 1,
\dots, N$, we have $P_i(\mathbf{Y}_j) = Q_i(\mathbf{X}_j)$, up to
homography in the targets. The two sets $\bxj$ and
$\{\mathbf{Y}_j\}$ are called {\it conjugate critical
configurations}, with {\it associated conjugate} projections
$\{Q_i\}$ and $\{P_i\}$.
\end{definition}

In next Proposition \ref{prop-2-1}, we prove that points in critical configurations fill an algebraic variety, called {\em critical locus}  $\mathcal{X}$, whose ideal
can be obtained by making use of the Grassmann tensor introduced
above.

Indeed, the Grassmann tensor $\mathcal{T}^{P_1,\dots,P_n}(L_1,
\dots, L_n)$ encodes the algebraic relations between corresponding
subspaces in the different views of the projections
$P_1,\dots,P_n$. Hence by definition of critical set, if
$\{\mathbf{X}_j,\mathbf{Y}_j\}$ are conjugate critical
configurations, then, for each $j$, the projections
$Q_1(\mathbf{X}_j), \dots, Q_n(\mathbf{X}_j)$ are corresponding
points not only for the projections $Q_1, \dots,Q_n,$ but for the
projections $P_1, \dots,P_n$, too.

Following the construction above, we first choose a profile
$(\alpha_1, \dots, \alpha_n)$, and a point $ \mathbf{X} $ in the
critical locus. If $ Q_i(\mathbf{X}) \in L_i$, for every $ i=1,
\dots, n$, then $\mathcal{T}^{P_1,\dots,P_n}(L_1, \dots, L_n) =
0$. The previous condition is fulfilled if $ L_i $ is spanned by $
Q_i(\mathbf{X}) $ and any other $ h_i-\alpha_i $ independent
points in $ \mathbb{P}^{h_i}$. So, we can suppose $$ S_i = \left(
\begin{array}{cccc} Q_i(\mathbf{X}) & \mathbf{x}_{i1} & \dots &
\mathbf{x}_{i,h_i-\alpha_i} \end{array} \right) = \left(
\begin{array}{cc} Q_i(\mathbf{X}) & S'_i \end{array} \right) $$
of maximal rank $ h_i-\alpha_i+1$, that is to say, $ S'_i $
is a general $ (h_i+1) \times (h_i-\alpha_i) $ matrix of rank $
h_i-\alpha_i$. Due to this choice, the matrix $ T^{P_1, \dots,
P_n}_{S_1, \dots, S_n} $ becomes
$$ T^{P_1, \dots,
P_n}_{S_1, \dots, S_n} = \left( \begin{array}{cccccccccc} P_1 &
Q_1(\mathbf{X}) & S'_1 & 0 & 0 & 0 & \dots & 0 & 0 & 0 \\ P_2 & 0
& 0 & Q_2(\mathbf{X}) & S'_2 & 0 & \dots & 0 & 0 & 0 \\ \vdots & &
& & & & & & & \vdots \\ P_n & 0 & 0 & 0 & 0 & 0 & \dots & 0 &
Q_n(\mathbf{X}) & S'_n \end{array} \right).$$ The determinant $
\det(T^{P_1,\dots, P_n}_{S_1, \dots, S_n}) $ is a sum of products of
maximal minors of $ S'_1, \dots, S'_n$, and
maximal minors of the matrix \begin{equation} \label{matrix-M}  M^{P_1, \dots,
P_n}_{Q_1, \dots, Q_n} = \left( \begin{array}{ccccccc} P_1 &
Q_1(\mathbf{X}) & 0 & 0 & \dots & 0 & 0 \\ P_2 & 0 &
Q_2(\mathbf{X}) & 0 & \dots & 0 & 0 \\ \vdots & & & & & & \vdots
\\ P_n & 0 & 0 & 0 & \dots & 0 & Q_n(\mathbf{X}) \end{array}
\right).\end{equation}
Such a matrix is
a $ (n+ \displaystyle \sum_{i=1}^n h_i) \times (n+k+1) $ matrix,
the last $n$ columns of which are of linear forms, while the first
$ k+1 $ columns are of constants.

More explicitly, if we consider $ M^{P_1, \dots,
P_n}_{Q_1, \dots, Q_n} $ as a block matrix, the coefficients are
the minors obtained by delating $ h_i-\alpha_i $ rows from the
$i$--th block $$ \left( \begin{array}{ccccccc} P_i & 0 & \dots & 0
& Q_i(\mathbf{X}) & 0 & \dots \end{array} \right),$$ for every
block.

If we allow the profile to change, because $ \mathbf{X} $ is in
the critical locus independently from the profile, we get all the
possible maximal minors of $ M^{P_1, \dots, P_n}_{Q_1, \dots,
Q_n}$. The discussion above is part of the proof of the following result.
\begin{Prop} \label{prop-2-1}
The ideal $ I(\mathcal{X})$ of the critical locus $ \mathcal{X} $
is generated by the maximal minors of $ M^{P_1, \dots, P_n}_{Q_1,
\dots, Q_n}$, and so $\mathcal{X}$ is a determinantal variety.
Moreover, $ \mathcal{X} $ contains the centers of the projections
$ Q_j$'s.
\end{Prop}

\begin{proof}
We only have to prove that the center $ C'_j $ of $ Q_j $ is
contained in $ \mathcal{X}$ for every $ j=1, \dots, n$. $ C'_j $
is the zero locus of $ Q_j(\mathbf{X})$, and so $ M^{P_1, \dots,
P_n}_{Q_1, \dots, Q_n} $ drops rank at every point in $ C'_j$, for
each $ j$.
\end{proof}

We remark that, if $ \alpha_i \geq 1$, then some maximal minors of
$ M^{P_1, \dots, P_n}_{Q_1, \dots, Q_n} $ do not appear in $
\det(T^{P_1,\dots, P_n}_{S_1, \dots, S_n}) $ for whatever profile, and so
they should not be among the generators of $ I(\mathcal{X})$. This
fact supports our choice to allow $ \alpha_i = 0$. \bigskip

From the first generality assumption, it follows
both that the first $ k+1 $ columns of $ M^{P_1, \dots, P_n}_{Q_1,
\dots, Q_n}$ are linearly independent, and that the linear forms
in the last $ n $ columns of the above matrix span a linear space
of dimension $ k+1 $ in $ R_1 = \left(R = K[x_0,\dots,
x_k]\right)_1$, where $ \mathbb{P}^k = \mbox{Proj}(R)$. In fact, no point is common to either the
centers of the $ P_i$'s or of the $ Q_j$'s. \bigskip

As in \cite{bnt1}, we write the matrix $ M^{P_1, \dots, P_n}_{Q_1, \dots, Q_n}$ as
the following block matrix $$ M^{P_1, \dots, P_n}_{Q_1, \dots,
Q_n} = \left( \begin{array}{cc} A & B \\ C & D \end{array} \right)
$$ where $ A $ is a $ (n -k-1+ \displaystyle \sum_{i=1}^n h_i) \times
(k+1)$ matrix, $ B $ is a $ (n -k-1+ \displaystyle \sum_{i=1}^n
h_i) \times n$ matrix, $ C $ is an order $ (k+1) $ square matrix,
and, finally, $ D $ is a $ (k+1) \times n $ matrix. We assume that
$ C $ is invertible. By performing elementary operations on
columns and rows, we can reduce $ M^{P_1, \dots, P_n}_{Q_1, \dots,
Q_n} $ to the following easier form $$ \left(
\begin{array}{cc} 0 & N_{\mathcal{X}} \\ I_{k+1} & 0 \end{array} \right)
$$ where $ N_{\mathcal{X}} = B - AC^{-1}D $ is a $ (n -k-1+ \displaystyle \sum_{i=1}^n
h_i) \times n$ matrix of linear forms. Furthermore, the maximal
minors of $ M^{P_1, \dots, P_n}_{Q_1, \dots, Q_n}$ span the same
ideal as the maximal minors of $ N_{\mathcal{X}}$. Hence, we have
the following result.
\begin{Cor} \label{N-mat-cor}
$ I(\mathcal{X}) $ is generated by the maximal minors of $
N_{\mathcal{X}} = B - A C^{-1} D$, with the same notations as
above.
\end{Cor}

Since the critical locus $ \mathcal{X} $ is a determinantal
variety whose ideal is generated by the maximal minors of a matrix
of linear forms, the expected dimension of $ \mathcal{X} $ is
\begin{equation} \label{exp-dim}
ed_{\mathcal{X}} = k - \left(1 + (n -k-1+ \sum_{i=1}^n h_i) -
n\right) = 2k - \sum_{i=1}^n h_i.
\end{equation}
From Porteous's formula
(\cite{acgh}, formula $4.2$, p. $86$),we get, if $ \dim(\mathcal{X})
= ed_{\mathcal{X}}$,
\begin{equation} \label{exp-deg} \deg(\mathcal{X}) =
\binom{n-k-1+\sum_{i=1}^n h_i}{n-1}.
\end{equation} \bigskip

{\em Second generality assumption:} we assume projections $ P_1, \dots, P_n$, and $ Q_1, \dots, Q_n$ are general enough to guarantee that the critical locus $ \mathcal{X} $ has the expected dimension $ 2k - \sum_{i=1}^n h_i$.

From now on, every time we assert we are in the general case, we assume that both the generality assumptions hold. \medskip


\section{Numerical bounds}
\label{numerical}
In this section, we deduce both a lower bound for $ \sum h_i$, and an upper bound on the number $ n $ of views to get smooth critical loci.
\begin{Prop}
In the same notations as above, we have \begin{equation} k+1 \leq
\sum_{i=1}^n h_i. \end{equation}
\end{Prop}

\begin{proof}
As $ 0 \leq \alpha_i \leq h_i $ for every $ i=1, \dots, n$,
and $ \sum_{i=1}^n \alpha_i = k+1$, then we have $ k + 1 \leq
\displaystyle \sum_{i=1}^n h_i$.
\end{proof}

We are interested in studying the case $ \mathcal{X} $ is irreducible and non--singular. To begin, we relate the projection centers to singular critical loci.
\begin{Lemma}
If two centers of the projections $ Q_1, \dots, Q_n $ intersect, the critical locus is singular.
\end{Lemma}

\begin{proof} If the centers of $ Q_1 $ and $ Q_2$, for example, have a common point, two columns of matrix (\ref{matrix-M}) vanish, and so its rank is at most $ k+n-1$. From generalities on determinantal varieties, it follows that the critical locus is singular.
\end{proof}

Now, we can compute an upper bound on the number $ n $ of views to get an associated smooth critical locus.

\begin{Thm} \label{bounds-on-n}
Let $ \mathcal{X} $ be the codimension $ c \geq 2 $ critical locus for a couple of $ n \geq 4 $ projections
$ P_1, \dots, P_n$ and $ Q_1, \dots, Q_n$ from $ \mathbb{P}^k$ to $ \mathbb{P}^{h_i}$, $ i=1, \dots, n$. Then, either $ \mathcal{X} $ is not irreducible, or is singular.
\end{Thm}

\begin{proof} The center $ C_i $ of $ Q_i $ has dimension $ k - h_i -1$. Since $ C_i \subseteq \mathcal{X}$, we have that $ \dim C_i \leq \dim \mathcal{X}$, and so $ k - h_i - 1 \leq k - c$, or equivalently, $$ h_i \geq c-1.$$

Let us assume that $ c-1 \leq h_1 \leq \dots \leq h_n$.

If $ h_1 = c-1$, then the center $ C_1 $ of $ Q_1 $ has codimension $ c $ and is contained in $ \mathcal{X}$. Then, $ \mathcal{X} $ is not irreducible.

Assume now that $ h_1 \geq c$. Since $ \dim \mathcal{X} = 2k - \sum_{i=1}^n h_i = k-c$, we get $$ k = \sum_{i=1}^n h_i -c.$$ On the other hand, the center $ C_ j $ of projection $ Q_j $ has dimension $$ \dim C_j = k - h_j -1 = \sum_{i \not= j} h_i - c -1.$$

We know that $ \mathcal{X} $ is singular if two centers meet. The two centers with smaller dimension are $ C_{n-1} $ and $ C_n$. The condition that guarantees they do not meet is $ \dim C_n + \dim C_{n-1} - k < 0$, that is to say, $$ \sum_{i =1}^{n-2} h_i < c+2.$$ Since we are in the case $ h_1 \geq c$, the left side becomes $ c(n-2) < c+2$, i.e. $ n < 3 + \frac 2c$. Hence, if $ h_1 \geq c $ and $ n \geq 4$, $ \mathcal{X} $ is singular.
\end{proof}

The same proof allows us to state also the following result.
\begin{Thm}
Let $ \mathcal{X} $ be the codimension $ 1 $ critical locus for a couple of $ n \geq 5 $ projections
$ P_1, \dots, P_n$ and $ Q_1, \dots, Q_n$ from $ \mathbb{P}^k$ to $ \mathbb{P}^{h_i}$, $ i=1, \dots, n$. Then, either $ \mathcal{X} $ is not irreducible, or is singular.
\end{Thm}

Hence, we have to study the cases $ n=2 $ and $ n=3$, for every codimension $ c $ and $ n=4 $ for $ c=1$, only.


\section{The $ n = 2 $ view case}
\label{n2}

In this section, we want to prove that, under some mild assumptions, the critical locus for $ n = 2 $ is a smooth and irreducible variety of minimal degree, and, conversely, that every smooth irreducible variety of minimal degree, but the Veronese surface, is critical for a suitable couple of projections. In this way, we classify all smooth critical loci in Computer Vision for $ 2 $ views. E.g., when the codimension is $1$, the critical locus for two projections from $ \mathbb{P}^3 $ to $ \mathbb{P}^2 $ is a quadric surface, and this is well--known in the Computer Vision community; when $c=2$, the critical locus for two projections from $ \mathbb{P}^4 $ to $ \mathbb{P}^3 $ is a rational normal scroll; when $c=3$ the critical locus for two projections from $ \mathbb{P}^5 $ to $ \mathbb{P}^4 $ is either $ \mathbb{P}(\mathcal{O}_{\mathbb{P}^1}(2) \oplus \mathcal{O}_{\mathbb{P}^1}(2))$, or $ \mathbb{P}(\mathcal{O}_{\mathbb{P}^1}(1) \oplus \mathcal{O}_{\mathbb{P}^1}(3))$. \medskip

When $ n=2$, we have $ h_1 + h_2 = k + c$. Moreover, $ k > h_2 \geq h_1 \geq c+1$. We remark that the last inequality on the right is a consequence of the previous equality.

Now, we prove that the critical locus is a variety of minimal degree.
\begin{Prop} \label{critical2minimal}
In the general case, the codimension $ c $ critical loci for two
views are minimal degree varieties.
\end{Prop}

\begin{proof} Corollary \ref{N-mat-cor} implies that the ideal $
I(\mathcal{X}) $ is generated by the maximal minors of $
N_{\mathcal{X}} $ whose type is $ (c+1) \times 2$, and so it is generated by quadrics. Furthermore, from equation (\ref{exp-deg}), we get that $ \deg(\mathcal{X}) = 1+c$. This description proves that $ \mathcal{X} \subseteq
\mathbb{P}^k $ is a minimal degree variety (see
\cite{Eisenbud-Harris}).
\end{proof}

The generality assumption in Proposition \ref{critical2minimal} implies that the minors of $ N_{\mathcal{X}} $ define variety of the expected codimension $ c$.

From the classification of minimal degree varieties in \cite{Eisenbud-Harris}, we get that $
\mathcal{X} $ is singular as soon as $ k \geq 2(c+1)$. Hence, smooth irreducible varieties of minimal degree that can be critical loci are embedded in $ \mathbb{P}^k $ for $ c+2 \leq k \leq 2c+1$.

Now, we consider the converse of Proposition \ref{critical2minimal}.
\begin{Prop}
With the only exception of Veronese surfaces in $ \mathbb{P}^5$, every codimension $ c $ minimal degree variety embedded in $ \mathbb{P}^k $ with $ c+2 \leq k \leq 2c+1$, is the critical locus for a suitable pair of projections.
\end{Prop}

\begin{proof} Let us consider matrix $ M $ in the case of $ 2 $ projections, and the matrix $ N_{\mathcal{X}} $ we obtain from it, as discussed in Section \ref{prelimsec}. In the $ 2 $ view case, we have $$ M = \left( \begin{array}{ccc} P_1 & Q_1(X) & 0 \\ P_2 & 0 & Q_2(X) \end{array} \right) = \left( \begin{array}{cc} A & B \\ C & D \end{array} \right),$$ where $ P_1 $ ($ P _2$, respectively) is a $ (h_1+1) \times (k+1) $ ($ (h_2+1) \times (k+1)$, respectively) full rank matrix, and, up to transposition, $ Q_1(X) = (Q_{11}(X), \dots, Q_{1,h_1+1}(X))$ and $ Q_2(X) = (Q_{21}(X), \dots, Q_{2,h_2+1}(X))$ and the linear forms in each one of them are linearly independent. Moreover, $ A $ is of type $ (c+1) \times (k+1)$, $ C $ is of type $ (k+1) \times (k+1)$, and we assume it is invertible, $ B $ is of type $ (c+1) \times 2$, and finally, $ D $ is of type $ (k+1) \times 2$. The assumption on the rank of $ C $ is always fulfilled up to collecting rows of $ M $ in a different way.

We recall that $ h_i \geq c+1$ for $ i=1,2$.

When computing $ N_{\mathcal{X}}$, we get $$ N_{\mathcal{X}} = \left( \left( \begin{array}{c} Q_{11}(X) \\ \vdots \\ Q_{1,c+1}(X) \end{array} \right) \right. \left. - E \left( \begin{array}{c} Q_{1,c+2}(X) \\ \vdots \\ Q_{1,h_1+1}(X) \end{array} \right) \right\vert \left. - F \left( \begin{array}{c} Q_{21}(X) \\ \vdots \\ Q_{2,h_2+1}(X) \end{array} \right) \right) $$ where $ A C^{-1} = (E \vert F)$. It is evident that the first column of $ N_{\mathcal{X}} $ depends on the first view, and the second from the second view.

Let us consider now a codimension $ c $ variety $ V $ of minimal degree, and let $ N $ be the $ (c+1) \times 2 $ matrix of linear forms associated to $ V $ (see \cite{Eisenbud-Harris}). To fix notation, let $ n_{ij} $ be the elements of $ N$. By comparing $ N $ and $ N_{\mathcal{X}}$, we can choose $ E $ and $ Q_{1,c+2}(X), \dots, Q_{1,h_1+1}(X) $ as general as possible, and we get $$ \left( \begin{array}{c} Q_{11}(X) \\ \vdots \\ Q_{1,c+1}(X) \end{array} \right) = \left( \begin{array}{c} n_{11} \\ \vdots \\ n_{c+1,1} \end{array} \right) + E \left( \begin{array}{c} Q_{1,c+2}(X) \\ \vdots \\ Q_{1,h_1+1}(X) \end{array} \right).$$ We choose $ F $ as $ (-I \vert F')$, $ F' $ being general, and $ Q_{2,c+2}(X), \dots, Q_{2,h_2+1}(X) $ arbitrary. Similarly to the previous case, we get $$ \left( \begin{array}{c} Q_{21}(X) \\ \vdots \\ Q_{2,c+1}(X) \end{array} \right) = \left( \begin{array}{c} n_{12} \\ \vdots \\ n_{c+1,2} \end{array} \right) + F' \left( \begin{array}{c} Q_{2,c+2}(X) \\ \vdots \\ Q_{2,h_2+1}(X) \end{array} \right).$$ Once we choose a general invertible matrix $ C $ of order $ k+1$, we compute $ A = (E \vert F) C $ and so we get the matrix $ M $ as required. We remark that the assumptions on the ranks of $ P_1, P_2, Q_1(X)$ and $ Q_2(X) $ are satisfied by the generality of the choices in the construction.
\end{proof}


\section{The $ n = 3 $ view case}
\label{n3}
In this section, we classify all smooth varieties that can be obtained as critical loci for two triples of projections.

Before approaching the problem, we briefly recall the list of codimension $ c $ smooth determinantal varieties associated to matrices of type $ (c+2) \times 3$. In the case under consideration, since there are three views, we have $ h_1 + h_2 + h_3 = k + c$. From the proof of Theorem \ref{bounds-on-n}, we know that, if a view verifies $ h_i \leq c-1$, then either the critical locus does not have codimension $ c$, or is not irreducible. Then, we can assume $ h_i \geq c $ for $ i=1,2,3$, from which we get that $ k \geq 2c$. On the other hand, a determinantal variety of codimension $ c $ as the ones we consider, is singular when embedded in $ \mathbb{P}^k $ with $ k \geq 2c+2$. Hence, smooth determinantal varieties can be critical loci for two triples of projection only if embedded in a projective space $ \mathbb{P}^k $ with $ k = 2c $ or $ k = 2c+1$. The degree of such varieties is the expected one, namely $ \deg(\mathcal{X}) = \binom{c+2}2$, as it follows from equation (\ref{exp-deg}). Thanks to classification results on smooth varieties with small invariants, in the case we are dealing with, the list of smooth varieties is complete for $ c \leq 3$. This is not a limitation for us because, as we will see in Theorem \ref{n=3-critical-loci}, smooth critical loci appear only when $ c \leq 2$.

For sake of completeness, we list the smooth determinantal varieties of degree $ \binom{c+2}2 $ for $ c \leq 3$:
\begin{enumerate}
\item[$c=1$] plane cubic curves: they can be critical for two triples of projections from $ \mathbb{P}^2 $ to $ \mathbb{P}^1$. Even if this case is not of interest for the Computer Vision community, we include it for completeness from a geometric perspective;
\item[$c=1$] cubic surfaces: they can be critical for two triples of projections from $ \mathbb{P}^3 $ to $ \mathbb{P}^1, \mathbb{P}^1, \mathbb{P}^2$;
\item[$c=2$] Bordiga surfaces: they can be critical for two triples of projections from $ \mathbb{P}^4 $ to $ \mathbb{P}^2$;
\item[$c=2$] Bordiga scrolls: they can be critical for two triples of projections from $ \mathbb{P}^5 $ to $ \mathbb{P}^2, \mathbb{P}^2, \mathbb{P}^3$;
\item[$c=3$] $3$-fold scrolls on $ \mathbb{P}^2$: they can be critical for two triples of projections from $ \mathbb{P}^6 $ to $ \mathbb{P}^3$;
\item[$c=3$] $4$-fold scrolls on $ \mathbb{P}^2$: they can be critical for two triples of projections from $ \mathbb{P}^7 $ to $ \mathbb{P}^3, \mathbb{P}^3, \mathbb{P}^4$.
\end{enumerate}

The codimension $ 1 $ cases are part of classical results on the classification of smooth hypersurfaces, the codimension $ 2 $ ones are in \cite{Io}, while the codimension $ 3 $ ones are in \cite{FL}. We briefly describe the codimension $ 2 $ and $ 3 $ varieties.

The Bordiga surface is the embedding in $ \mathbb{P}^4 $ of the blow--up of $ \mathbb{P}^2 $ at $ 10 $ general points via the linear system of plane quartics through the points. Let $ Z \subseteq \mathbb{P}^2 $ be a set of $ 10 $ general points, and $ B \subseteq \mathbb{P}^4 $ the associated Bordiga surface. The ideal sheaves $ \mathcal{I}_Z $ and $ \mathcal{I}_B $ are described from the following exact sequences: $$ 0 \to \mathcal{O}_{\mathbb{P}^2}^4(-5) \stackrel{N_Z}{\longrightarrow} \mathcal{O}_{\mathbb{P}^2}^5(-4) \to \mathcal{I}_Z \to 0 \quad \mbox{ and } \quad 0 \to \mathcal{O}_{\mathbb{P}^4}^3(-4) \stackrel{N_B}{\longrightarrow} \mathcal{O}_{\mathbb{P}^4}^4(-3) \to \mathcal{I}_B \to 0.$$ The matrices $ N_Z $ and $ N_B $ are not independent since it holds \begin{equation} \label{NZNB} (x_0, \dots, x_4) N_Z = (z_0, z_1, z_2) N_B^T \end{equation} where $ x_0, \dots, x_4 $ are coordinates in $ \mathbb{P}^4 $ and $ z_0, z_1, z_2 $ are coordinates in $ \mathbb{P}^2$.

The Bordiga scroll $ X $ is $ \mathbb{P}(\mathcal{E}) $ embedded in $ \mathbb{P}^5 $ via the tautological bundle $ \xi$, where $ \mathcal{E} $ is any rank $ 2 $ vector bundle defined by the extension $$ 0 \to \mathcal{O}_{\mathbb{P}^2} \to \mathcal{E} \to \mathcal{I}_Z(4) \to 0 $$ where $ Z \subseteq \mathbb{P}^2 $ is a set of $ 10 $ general points, as for the Bordiga surface (see \cite{ottaviani-scrolls}). By comparing the resolution of $ \mathcal{I}_Z $ and the defining extension for $ \mathcal{E}$, we get the following exact sequence $$ 0 \to \mathcal{O}_{\mathbb{P}^2}^4(-1) \stackrel{N_{\mathcal{E}}}{\longrightarrow} \mathcal{O}_{\mathbb{P}^2}^6 \to \mathcal{E} \to 0.$$ The minimal free resolution of $ X $ is $$ 0 \to \mathcal{O}_{\mathbb{P}^5}^3(-4) \stackrel{N_X}{\longrightarrow} \mathcal{O}_{\mathbb{P}^5}^4(-3) \to \mathcal{I}_X \to 0,$$ and, as for the Bordiga surface, the matrices $ N_X $ and $ N_{\mathcal{E}} $ are related in the equation $$ (x_0, \dots, x_5) N_{\mathcal{E}} = (z_0, z_1, z_2) N_X^T.$$

In the codimension $ 3 $ case, the resolutions of the two scrolls are obtained by means of the Eagon--Northcott complex, and it follows that they both have sectional genus $ 6$. In such a case, we can construct them both similarly to the case of the Bordiga scroll. We remark that, in principle, the $3$--fold scroll could be constructed as blow--up of a scroll at four double points, but it is not known whether it exists. The starting point is now a set $ Z $ of $ 15 $ general points in $ \mathbb{P}^2$, and a rank $ 2 $ vector bundle $ \mathcal{E} $ defined by the extension $$ 0 \to \mathcal{O}_{\mathbb{P}^2} \to \mathcal{E} \to \mathcal{I}_Z(5) \to 0.$$ Since the minimal free resolution of $ \mathcal{I}_Z $ is $$ 0 \to \mathcal{O}_{\mathbb{P}^2}^5(-6) \to \mathcal{O}_{\mathbb{P}^2}^6(-5) \to \mathcal{I}_Z \to 0,$$ we get the following presentation of $ \mathcal{E}$: $$ 0 \to \mathcal{O}_{\mathbb{P}^2}^5(-1) \stackrel{N_{\mathcal{E}}}{\longrightarrow} \mathcal{O}_{\mathbb{P}^2}^7 \to \mathcal{E} \to 0.$$ Let $ N_X $ be the $ 5 \times 3 $ matrix that satisfy the equation $$ (x_0, \dots, x_6) N_{\mathcal{E}} = (z_0, z_1, z_2) N_X^T.$$ Then, the defining ideal $ I(X) $ of $ X $ is generated by the $ 3 \times 3 $ minors of $ N_X$.

Since the construction of the $ 4$--fold scroll is analogous to the one of the $3$--fold scroll, we only stress the differences. This time, the vector bundle to consider is the rank $ 3 $ one defined by the extension $$ 0 \to \mathcal{O}_{\mathbb{P}^2}^2 \to \mathcal{E} \to \mathcal{I}_Z(5) \to 0,$$ where $ Z $ is a set of $ 15 $ general points as in the previous case. Then, the matrix $ N_X $ is obtained as in the previous case, but $ N_{\mathcal{E}} $ is now a matrix with type $ 8 \times 5$.

Now, we address the problem of getting the above smooth varieties as critical loci for the reconstruction problem. As we have seen above, we consider a codimension $ c$, determinantal variety $ X \subseteq \mathbb{P}^k$, with $ k = 2c$ or $ k = 2c+1$, whose defining ideal is generated by the $ 3 \times 3 $ minors of a $ (c+2) \times 3 $ matrix $ N $ of linear forms. The result is contained in the following classification Theorem.

\begin{Thm} \label{n=3-critical-loci}
$ X $ is the critical locus for two suitable triples of projections from $ \mathbb{P}^k $ if and only if either $ X \subseteq \mathbb{P}^2 $ is a cubic curve, or $ X \subseteq \mathbb{P}^3 $ is a cubic surface, or, finally, $ X \subseteq \mathbb{P}^4 $ is a Bordiga surface. In particular, $ c \leq 2$.
\end{Thm}

\begin{proof} From the previous discussion, it follows that $ h_1 = h_2 = c$, and $ h_3 = c + \varepsilon $ for $ k = 2c + \varepsilon$, $ \varepsilon = 0 $ or $ 1$. Let us consider general projections $ P_i, Q_i : \mathbb{P}^k \to \mathbb{P}^{h_i} $ for $ i=1,2,3$.

At first, we compute the matrix $ N_{\mathcal{X}} $ as in Corollary \ref{N-mat-cor}, and so we get $$ N_{\mathcal{X}} = \left( \begin{array}{ccc} Q_{11}(X) & 0 & 0 \\ \vdots \\ Q_{1,c+1}(X) & 0 & 0 \\ 0 & Q_{21}(X) & 0 \end{array} \right) - A C^{-1} \left( \begin{array}{ccc} 0 & Q_{22}(X) & 0 \\ \vdots \\ 0 & Q_{2,c+1}(X) & 0 \\ 0 & 0 & Q_{31}(X) \\ \vdots \\ 0 & 0 & Q_{3,h_3+1}(X) \end{array} \right).$$ If we multiply the matrices above, and perform elementary operations on the rows, we get the matrix $ N = (n_{ij}) $ such that: $(i)$ its $ j$-th column depend on $ Q_j $ only, for $ j=1,2,3$; $(ii)$ $ n_{c+2,1} = n_{c+1,2} = 0$. In the case $ h_3 = c$, it is possible to perform elementary operations of the rows of $ N $ so that $(iii)$ $ n_{c,3} = 0$.

It follows that the critical locus is actually a codimension $ c $ scheme whose defining ideal is generated by the $ 3 \times 3 $ minors of a matrix with type $ (c+2) \times 3 $ of linear forms.

It is known that the defining equation of every smooth cubic surface $ X \subseteq \mathbb{P}^3 $ can be written as $ L_1L_2L_3 + M_1M_2M_3 = 0 $ for suitable linear forms $ L_i, M_i$, $ i=1,2,3$. An equation of this kind for the cubic surface is called Cayley--Salmon. We remark that, for a given surface, there are $120$ different Cayley--Salmon equations that define it (see \cite{hlv}). The lines defined by $ L_i = M_j = 0 $ are contained in the cubic surface $ X$. The Cayley--Salmon equation is the locus where $$ N_X = \left( \begin{array}{ccc} L_1 & M_2 & 0 \\ M_1 & 0 & L_3 \\ 0 & L_2 & M_3 \end{array} \right) $$ drops rank. If we add scalar multiples of the first two columns to the third one, we get a matrix $ N_X $ that verifies constrains $ (i), (ii) $ above, and the linear forms on the third column are linearly independent. If we compare matrices $ N_{\mathcal{X}} $ and $ N_X$ under the simplifying assumption that $$ (E \vert F) = -\left( \begin{array}{cccc} 1 & 1 & 0 & 0 \\ 0 & 0 & 1 & 0 \\ -e_{31} & 0 & 0 & 1 \end{array} \right),$$ we get $ Q_{11} = L_1, Q_{12} = M_1, Q_{21} = e_{31} M_2 + L_2, Q_{22} = M_2$, and $ Q_{3j} = n_{j3} $ for $ j=1,2,3$, and so we obtain the projections $ Q_1, Q_2, Q_3$. If we choose a general invertible matrix $ C $ of order $ 4$, we get $ A = (E \vert F) C$, and so we obtain also the projections $ P_1, P_2, P_3$.

We have then proven that every smooth cubic surface is the critical locus for two triples of projections from $ \mathbb{P}^3 $ to $ \mathbb{P}^1, \mathbb{P}^1, \mathbb{P}^2$, as claimed.

A plane cubic curve is obtained as a section of a smooth cubic surface with a general plane. Hence, the argument above shows also that every plane cubic curve is critical for two triples of projections from $ \mathbb{P}^2 $ to $ \mathbb{P}^1$, as claimed. We remark that, for plane cubic curves, one has to start from the Cayley--Salmon equation and does not have to perform further elementary operations on the columns of the matrix.

The Bordiga surface has been considered from the point of view of critical loci in \cite{bnt1}. In that paper, the authors proved that the critical locus for two triples of projections from $ \mathbb{P}^4 $ to $ \mathbb{P}^2 $ is in the irreducible component of the Hilbert scheme containing the Bordiga surface as general element (\cite{bnt1}, Proposition 5.1), and conversely, that every Bordiga surface $ X $ is actually critical for two suitable triples of projections (\cite{bnt1}, Theorem 5.1). The key point of the proof of Theorem 5.1 is that, if the unit points in $ \mathbb{P}^2 $ are in $ Z$, then the matrix $ N_X $ fulfils the constraints $(i), (ii), (iii)$ above, since it is related to $ N_Z $ in equation (\ref{NZNB}).

To complete the proof, we'll prove that the critical locus $ \mathcal{X} $ is never smooth in the remaining cases.

To this end, we consider a general critical locus $ \mathcal{X}$, its associated matrix $ N_{\mathcal{X}} $ as obtained at the beginning of the proof, and we take the codimension $ c+3 $ linear space $ L \subseteq \mathcal{X} $ defined by $ n_{11} = \dots = n_{c+1,1} = n_{c+2,2} = n_{c+2,3} = 0$. As $ \mathcal{X} $ is at least $ 3$--dimensional in the cases we are considering, $ L $ is not empty. To prove that the points in $ L $ are singular for $ \mathcal{X}$, we evaluate the Jacobian matrix at them. Without loss of generality, we can make a change of coordinates, so that $ n_{i,1} = x_i $ for $ i=1, \dots, c+1$, $ n_{c+2,2} = x_{c+2} $ and $ n_{c+2,3} = x_{c+3}$. Furthermore, we assume that, at a point in $ L$, the rank of the $ (c+1) \times 2 $ matrix obtained by removing the first column and the last row in $ N_{\mathcal{X}} $ is $ 2$. In the case this does not hold, the point is singular by general properties of determinantal varieties. To simplify notation, we denote $ (i_1i_2) $ the determinant of the minor of the above matrix obtained by taking rows $ i_1 $ and $ i_2$. Let $ f_{ijh} $ be the determinant of the submatrix of $ N_{\mathcal{X}} $ obtained by taking rows $ i, j, h $ with $ 1 \leq i < j < h \leq c+2$. The derivative of $ f_{ijh} $ with respect to any variable is the sum of $ 3 $ determinants, two columns of which are from $ N_{\mathcal{X}} $ and the third column is the derivative of the corresponding column in $ N_{\mathcal{X}}$. We have to evaluate the derivatives at $ P \in L$. If $ h = c+2$, then the gradient of $ f_{ij,c+2} $ at $ P $ is the null matrix, as it is easy to check. If $ h \leq c+1$, we get $$ \nabla f_{ijh}(P) = (jh) \vec{e}_i - (ih) \vec{e}_j + (ij) \vec{e}_h $$ where $ \vec{e}_k $ is the $ k$--th element of the canonical basis. Without loss of generality, we can assume $ (12) \not= 0$, equivalent to the rank two assumption. The matrices $ \nabla f_{123}(P), \dots,$ $ \nabla f_{12,c+1}(P) $ are linearly independent, since $ (12) I_{c-1} $ is a submatrix of the Jacobian matrix corresponding to the above generators. Let us consider now $ f_{ijh} $ with $ 2 < j < h \leq c+1$ and $ i=1 $ or $ i=2$. We have $$ (12) \nabla f_{ijh}(P) - (ij) \nabla f_{12h}(P) + (ih) \nabla f_{12j}(P) = [(12)(jh) - (1j)(2h) + (1h)(2j)] \vec{e}_i = 0 $$ because the equation in square brackets is a Pl\"ucker relation that holds for rank $ 2 $ matrices of type $ (c+1) \times 2 $ for every $ c \geq 3$. Finally, we consider $ f_{ijh} $ with $ 2 < i < j < h \leq c+1$. We have \begin{equation*} \begin{split} (12) & \nabla f_{ijh}(P) - (ij) \nabla f_{12h} + (ih) \nabla f_{12j}(P) - (jh) \nabla f_{12i}(P) =\\ = -[(2i)(jh)& - (2j)(ih) + (2h)(ij)] \vec{e}_1 + [(1i)(jh) - (1j)(ih) + (1h)(ij)] \vec{e}_2 = 0 \end{split} \end{equation*} because we get Pl\"ucker relations once more. Hence, $ \nabla f_{ijh}(P) $ is in the span of $ \nabla f_{123}(P),$ $ \dots, \nabla f_{12,c+1}(P) $ for every $ 1 \leq i < j < h \leq c+2$, and so the Jacobian matrix has rank $ c-1 $ at most at every $ P \in L$. This proves that every $ P \in L $ is singular for the critical locus $ \mathcal{X}$, and so the proof is complete.
\end{proof}


\section{The $ n = 4 $ view case}
\label{n4}
According to Theorem 2.2, when we have $ 4 $ views, the codimension of the critical locus is $ 1$, otherwise the critical locus is either not irreducible or singular. In such a case, $ h_1 + h_2 + h_3 + h_4 = k+1$. On the other hand, a degree $ 4$, determinantal hypersurface is singular if embedded in $ \mathbb{P}^k $ with $ k \geq 4$. Hence, the only possible case is $ k=3$, and $ h_i = 1 $ for every $ i=1, \dots, 4$. As previously said, this case is not of interest for the Computer Vision community, and we insert it for sake of completeness from a geometrical point of view.

The study of quartic determinantal surfaces in $ \mathbb{P}^3 $ is a classical topic, and we briefly recall the main results (see \cite{jessop-book} for more results on the subject).

Quartic surfaces in $ \mathbb{P}^3 $ are parameterized by points in $ \mathbb{P}^{34}$. It is known that the general quartic surface in $ \mathbb{P}^3 $ is not determinantal, and that the locus of determinantal ones is a divisor in $ \mathbb{P}^{34}$. Determinantal quartic surfaces are characterized as the ones that contain a non--hyperelliptic curve $ C $ of degree $ 6 $ and genus $ 3$ (see, e.g., \cite{beauville}). Such a curve is also called Schur's sextic.

Moreover, a general quartic surface does not contain any line. It is known that not ruled quartic surfaces can contain any number of lines in the range $ 1 $ to $ 52$, or $ 54, 56, 60, 64$ lines. In \cite{room2, room3, room4}, the author studied quartic determinantal surfaces containing one or two lines. In particular, if $ N $ is an order $ 4 $ matrix of linear forms in $ \mathbb{P}^3$ whose determinant is the defining equation of the quartic surface $ S$, and $ \ell \subset S $ is a line, then, up to elementary operations on rows and columns of $ N$, the linear forms defining $ \ell $ are either in a row or column of $ N$, or in a $ 3 \times 2 $ or $ 2 \times 3 $ submatrix of $ N$.

Now we discuss the connections between quartic determinantal surfaces containing lines and the reconstruction problem in Computer Vision.
\begin{Prop} \label{7-1}
Let $ P_i, Q_i: \mathbb{P}^3 \to \mathbb{P}^1, i = 1, \dots, 4$, be two $4$--tuples of projections. Then, in the general case, the associated critical locus is a smooth quartic determinantal surface.
\end{Prop}

\begin{proof} The matrix $ M $ associated to the two $4$--tuples of projections is described in equation (\ref{matrix-M}), and is a square matrix of order $ 8$. From $ M$, we get matrix $ N_{\mathcal{X}} = B - A C^{-1} D$ of order $ 4 $ whose determinant defines the critical locus $ \mathcal{X}$. In the considered case, matrices $ B, D$ are $$ B = \left( \begin{array}{cccc} Q_{11}(X) & 0 & 0 & 0 \\ Q_{12}(X) & 0 & 0 & 0 \\ 0 & Q_{21}(X) & 0 & 0 \\ 0 & Q_{22}(X) & 0 & 0 \end{array} \right), \qquad D = \left( \begin{array}{cccc} 0 & 0 & Q_{31}(X) & 0 \\ 0 & 0 & Q_{32}(X) & 0 \\ 0 & 0 & 0 & Q_{41}(X) \\ 0 & 0 & 0 & Q_{42}(X) \end{array} \right).$$ Then, the first two columns of $ N_{\mathcal{X}} $ are the first two columns of $ B$, while the last two columns of $ N_{\mathcal{X}} $ depend of the two non--zero columns of $ D$. Then, the critical locus is a quartic determinantal surface. When computing $ \mathcal{X} $ in a random case, we get a smooth surface, and so the general critical locus is smooth.
\end{proof}

\begin{rem}
In the notation of \cite{room2}, the four lines, centers of projections $ Q_1, \dots, Q_4$, are of type $4''$.
\end{rem}

Now we highlight a geometrical property of such critical loci.
\begin{Prop}
In the same hypotheses as above, the critical locus contains twisted cubic curves meeting three of the four lines at two points.
\end{Prop}

\begin{proof} Let $ N_{\mathcal{X}} $ be the matrix constructed in the proof of Proposition \ref{7-1}, and let $ N'$ its sumatrix consisting of the first three columns. The maximal minors of $ N'$ define a Schur curve containing the centers of projections $ Q_1, Q_2, Q_3$ as components. To fix notation, we set $$ N' = \left( \begin{array}{ccc} Q_{11}(X) & 0 & n'_{13} \\ Q_{12}(X) & 0 & n'_{23} \\ 0 & Q_{21}(X) & n'_{33} \\ 0 & Q_{22}(X) & n'_{43} \end{array} \right) $$ where the center of $ Q_1 $ is the line $ Q_{11}(X) = Q_{12}(X) = 0$, the center of $ Q_2 $ is the line $ Q_{21}(X) = Q_{22}(X) = 0 $ and the center of $ Q_3 $ is $ n'_{13} = n'_{23} = 0$, also defined by $ n'_{33} = n'_{43} = 0$. The two couples of generators of the third line are related by the equation $$ \left( \begin{array}{c} n'_{33} \\ n'_{43} \end{array} \right) = \left( \begin{array}{cc} a_{11} & a_{12} \\ a_{21} & a_{22} \end{array} \right) \left( \begin{array}{c} n'_{13} \\ n'_{23} \end{array} \right) $$ where $ A = (a_{ij}) $ is invertible. The residual curve is the twisted cubic curve defined by the $ 2 \times 2 $ minors of $$ \left( \begin{array}{c} Q_{11}(X) \\ Q_{12}(X) \end{array} \right\vert\left. \mbox{Adj}(A) \left( \begin{array}{c} Q_{21}(X) \\ Q_{22}(X) \end{array} \right) \right\vert\left. \begin{array}{c} n'_{13} \\ n'_{23} \end{array} \right),$$ as it can be checked. Each center meets the twisted cubic curve above at two points, because the generators of each line vanish two quadrics of the three that generate the twisted cubic.
\end{proof}

Now, we consider a partial converse of the above Proposition.
\begin{Prop}
Let $ \ell_1, \dots, \ell_4 \subset \mathbb{P}^3 $ be $ 4 $ lines, pairwise skew. Then there is a quartic determinantal surface $ S $ containing the $ 4 $ lines that is critical for two $4$--tuples of projections from $ \mathbb{P}^3 $ to $ \mathbb{P}^1$. The given lines are centers for the four projections $ Q_1, \dots, Q_4$.
\end{Prop}

\begin{proof} Let $ \ell_i $ be the line defined by $ I(\ell_i) = \langle q_{i1}, q_{i2} \rangle$, where $ q_{i1}, q_{i2} $ are linearly independent linear forms in $ \mathbb{C}[x_0, \dots, x_3]$. Let us consider the matrix $$ N = \left( \begin{array}{cccc} q_{11} & 0 & e_{11} q_{31} + e_{12} q_{32} & e_{13} q_{41} + e_{14} q_{42} \\ q_{12} & 0 & e_{21} q_{31} + e_{22} q_{32} & e_{23} q_{41} + e_{24} q_{42} \\ 0 & q_{21} & e_{31} q_{31} + e_{32} q_{32} & e_{33} q_{41} + e_{34} q_{42} \\ 0 & q_{22} & e_{41} q_{31} + e_{42} q_{32} & e_{43} q_{41} + e_{44} q_{42} \end{array} \right),$$ and let $ S $ be the surface defined by $ \det(N) = 0$. We assume $ E = (e_{ij}) $ to be invertible. Hence, we can reconstruct $ M $ from $ N $ by choosing a general invertible matrix $ C $ and by setting $ A = -E C$. Hence, $ S $ is critical for suitable projections as claimed.
\end{proof}

To complete the section, we make some final remarks on quartic surfaces that contain four skew lines that explain why it is not possible to give a stronger converse of Proposition 5.1.
\begin{rem}
Given $ 4 $ pairwise skew general lines in $ \mathbb{P}^3$, the linear subspace $ V $ of $ \mathbb{C}^{35} $ containing quartic surfaces through the lines has dimension $ 15$. Here, four lines are general if they are not contained in the same quadric. Since the conditions of being determinantal and of containing lines are independent, we expect that the quartic determinantal surfaces containing the four lines are a locus of dimension $ 14$ in $ V$.

From a parameter count, quartic determinantal surfaces that are critical loci depend on the elements of $ E$. Since we can get the same quartic surface for different choices of matrix $ E$ (see the proof of Proposition 5.2), we have that the parameters are actually less than $ 16$. Algebraically, from every rows of $ E$, different from the first one, an element can be disregarded. So, the locus in $ V $ of critical surfaces has dimension $ 13$. From a geometrical point of view, a quartic determinantal surface containing $ 4 $ lines as above is critical if, and only if, it contains a twisted cubic curve meeting three of the four lines at $ 2 $ points, and not meeting the last line, for every choice of three among the four lines. Since a twisted cubic curve is uniquely determined by $ 6 $ points, we get such a curve by choosing two points of the first three lines. If we choose a second twisted cubic curve meeting all lines but the third one, we need $ 6 $ more points, two for each of the three selected lines. Once those two twisted cubics are selected, the quartic surface through the four lines is given, and it is possible to get its defining equation as determinant of a matrix as in the proof of Proposition 5.2. The parameters from which this construction depends are $ 12$ (the points on the lines), and one more because we can multiply the matrix by a scalar so that the determinant defines the same surface. Hence, we get once more a locus of dimension $ 13$.

In conclusion, we expect the critical quartic surfaces to fill a codimension $ 2 $ subset in $ V$.
\end{rem}



\begin{thebibliography}{10}

\bibitem{acgh}
E.~Arbarello, M.~Cornalba, P.~Griffiths, J.D.~Harris.
\newblock Geometry of Algebraic Curves.
\newblock {\bf GMW 267}, Springer Verlag, New York, 1985.

\bibitem{AstromKahl}
K.~{\AA}str{\"{o}}m and F.~Kahl.
\newblock Ambiguous Configurations for the $1D$ Structure and
Motion Problem.
\newblock {\em Journal of Mathematical Imaging and Vision}, 18:191--203, 2003.

\bibitem{beauville}
A.~Beauville.
\newblock Determinantal Hypersurfaces.
\newblock {\em Michigan Math. J.} {\bf 48} (2000), 39--64.

\bibitem{tubbCHAPTER}
M.~Bertolini, G.~Besana, and C.~Turrini.
\newblock Applications Of Multiview Tensors In Higher Dimensions.
\newblock In {\em Tensors in Image Processing and Computer Vision}, Editors:
Aja-Fernandez, S., de Luis Garcia, R., Tao, D., Li, X.,
\newblock {\em Advances in Pattern Recognition}. 2009.

\bibitem{LAIA}
M.~Bertolini, G.~Besana, and C.~Turrini.
\newblock Critical loci for projective reconstruction from multiple views in higher dimension: A comprehensive theoretical approach.
\newblock {\em Linear Algebra and its Applications}, 469(2015), 335--363.

\bibitem{BertoliniAl2016}
M.~Bertolini, G.~Besana, and C.~Turrini.
\newblock Generalized Fundamental Matrices as Grassmann Tensors
\newblock {\em Annali di Matematica Pura ed Applicata}, n.2, 196(2017), 539--553.

\bibitem{bbnt}
M.~Bertolini, GM.~Besana, R.~Notari and C.~Turrini.
\newblock Critical loci in computer vision and matrices dropping rank in codimension one.
\newblock {\em Journal of Pure and Applied Algebra} (12) {\bf 224}, (2020), 1--27.

\bibitem{BM}
M.~Bertolini and L.~Magri.
\newblock Critical Hypersurfaces and Instability for Reconstruction of Scenes in High Dimensional Projective Spaces.
\newblock {\em Machine Graphics \& Vision}, (1) {\bf 29}, (2020), 3--20.

\bibitem{bnt1}
M.~Bertolini, R.~Notari and C.~Turrini.
\newblock The Bordiga surface as critical locus for 3-view reconstructions.
\newblock MEGA 2017, Effective Methods in Algebraic Geometry, Nice (France), June 12-16, 2017.
\newblock {\em J. Symb. Comp.} \bf 91\rm, (2019), 74--97.

\bibitem{be-tur1}
M.~Bertolini and C.~Turrini.
\newblock Critical configurations for 1-view in projections from
  $\mathbb{P}^{k} \to \mathbb {P}^{2}$.
\newblock {\em Journal of Mathematical Imaging and Vision}, 2007.

\bibitem{buch-88}
T.~Buchanan.
\newblock The twisted cubic and camera calibration.
\newblock {\em Comput. Vision Graphics Image Process.}, 42(1):130--132, 1988.

\bibitem{Eisenbud-Harris}
D.~Eisenbud, J.~Harris.
\newblock On Varieties of Minimal Degree (A Centennial Account).
\newblock {\em Proceedings of Symposia in Pure Mathematics}, vol. {\bf 46} (1987), pp. 3--13.

\bibitem{FL}
M.L.~Fania, E.L.~Livorni.
\newblock Degree Ten Manifolds of Dimension $ n $ Greater than or Equal to $ 3$.
\newblock {\em Math. Nachr.} {\bf 188} (1997), 79--108.

\bibitem{hlv}
M.A.~Hahn, S.~Lamboglia, A.~Vargas.
\newblock A short note on Cayley--Salmon equations.
\newblock {\em Le Matematiche}, vol. LXXV, Issue II (2020), 559--574.

\bibitem{Hart1}
R.I.~Hartley.
\newblock Ambiguous configurations for 3-view projective reconstruction.
\newblock In {\em European Conference on Computer Vision}, pages I: 922--935,
  2000.

\bibitem{Hart-Ka}
R.I.~Hartley and F.~Kahl.
\newblock Critical configurations for projective reconstruction from multiple
 views.
\newblock {\em International Journal of Computer Vision}, 71(1):5--47, 2007.

\bibitem{Hart-Schaf}
R.I.~Hartley and F.~Schaffalitzky.
\newblock Reconstruction from projections using {G}rassmann tensors.
\newblock In {\em Proceedings of the 8th European Conference on Computer
  Vision, Prague, Czech Republic}, LNCS. Springer, 2004.

\bibitem{vid4}
R.~Hartley and R.~Vidal.
\newblock The multibody trifocal tensor: Motion segmentation from $3$ perspective
  views.
\newblock In {\em Conference on Computer Vision and Pattern Recognition}, 2004.

\bibitem{Hart-Zis}
R.I.~Hartley and A.~Zisserman.
\newblock Multiple View Geometry in Computer Vision.
\newblock Cambridge University Press, 2004.

\bibitem{Io}
P.~Ionescu.
\newblock Embedded projective varieties of small invariants.
\newblock Proceeding of $1982$ Week of Algebraic Geometry, Bucharest.
\newblock LNM \bf 1056 \rm (1984), 142--187.

\bibitem{jessop-book}
C.M.~Jessop.
\newblock Quartic surface.
\newblock Cambridge University Press, 1916.

\bibitem{KHA}
F.~Kahl, R.~Hartley, and K.~{\AA}str{\"o}m.
\newblock Critical configurations for n-view projective reconstruction.
\newblock In {\em IEEE Computer Society Conference on Computer Vision and
 Pattern Recognition}, pages II:158--163, 2001.

\bibitem{kra}
J.~Krames.
\newblock Zur ermittlung eines objectes aus zwei perspectiven (ein beitrag zur
 theorie der ``gefh{\"a}hlichen {\"O}rter").
\newblock {\em Monatsh. Math. Phys.}, 49:327--354, 1940.

\bibitem{May-92}
S.J.~Maybank.
\newblock Theory of Reconstruction from Image Motion.
\newblock Springer-Verlag New York, Inc., Secaucus, NJ, USA, 1992.

\bibitem{ottaviani-scrolls}
G.~Ottaviani.
\newblock On $ 3$--folds in $ \mathbb{P}^5 $ which are scrolls.
\newblock {\em Annali Scuola Normale Pisa}, Vol. 19, n.3 (1992), 451--471.

\bibitem{room2}
T.G.~Room.
\newblock Self-transformations of determinantal quartic surfaces. II.
\newblock {\em Proc. London Math. Soc.} (2) 51 (1950), 362--382.

\bibitem{room3}
T.G.~Room.
\newblock Self-transformations of determinantal quartic surfaces. III.
\newblock {\em Proc. London Math. Soc.} (2) 51 (1950), 383--387.

\bibitem{room4}
T.G.~Room.
\newblock Self-transformations of determinantal quartic surfaces. IV.
\newblock {\em Proc. London Math. Soc.} (2) 51 (1950), 388--400.

\bibitem{May-Shas2}
A.~Shashua and S.J.~Maybank.
\newblock Degenerate $n$ point configurations of three views: Do critical
 surfaces exist?
\newblock TR 96-19, Hebrew University, 1996.

\bibitem{Shas-Wo}
L.~Wolf and A.~Shashua.
\newblock On projection matrices $\mathbb{P}^k \to \mathbb{P}^2, k=3,..., 6$
  and their applications in computer vision.
\newblock In {\em International Journal of Computer Vision}, 2002.

\end{thebibliography}
\end{document}